\documentclass{article}
\usepackage{amssymb, amsmath, amsthm, bbm, mathtools, graphicx}
\usepackage[top=3cm, bottom=3cm, left=3cm, right=3cm]{geometry}

\usepackage{hyperref}
\usepackage{xcolor}
\hypersetup{
  colorlinks   = true, 
  urlcolor     = blue, 
  linkcolor    = red, 
  citecolor   = blue 
}
\usepackage{color}

\usepackage[round]{natbib}

\newtheorem{theorem}{Theorem}[section]
\newtheorem{lemma}{Lemma}[section]
\newtheorem{proposition}{Proposition}[section]

\newtheorem{remark}{Remark}

\theoremstyle{definition}
\newtheorem{definition}{Definition}[section]

\defcitealias{stanag4487}{NATO, 2009}
\defcitealias{stanag4488}{NATO, 2009}
\defcitealias{stanag4489}{NATO, 1999}

\defcitealias{mil-std-331D}{U.S.~Department of Defense, 2017}

\defcitealias{burke2017binary}{Burke and Truett, 2017}
\defcitealias{baker2021overview}{Baker, 2021}

\providecommand{\keywords}[1]
{
  \small	
  \textbf{\textit{Keywords---}} #1
}

\newcommand{\one}{\mathbbm{1}}
\newcommand{\leqst}{\leq_{\mathrm{st}}}

\newcommand{\prob}{\mathbb{P}}
\newcommand{\reals}{\mathbb{R}}
\newcommand{\given}{\,|\,}

\newcommand{\diff}{\text{d}}

\DeclareMathOperator*{\argmax}{arg\,max}

\title{Revisiting the Langlie procedure}
\author{Dennis Christensen$^{1,2}$ \\ \texttt{dchristensen@nr.no}}

\date{$^1$Norwegian Computing Center (NR) \\
    $^2$Norwegian Defence Research Establishment (FFI)}

\begin{document}
    \maketitle
    
\begin{abstract}
  Introduced in 1962, the Langlie procedure is one of the most popular approaches to sensitivity testing. It aims to estimate an unknown sensitivity distribution based on the outcomes of binary trials. Officially recognized by the U.S.~Department of Defense, the procedure is widely used both in civil and military industry. It first provides an experimental design for how the binary trials should be conducted, and then estimates the sensitivity distribution via maximum likelihood under a simple parametric model like logistic or probit regression. Despite its popularity and longevity, little is known about the statistical properties of the Langlie procedure, but it is well-established that the sequence of inputs tend to narrow in on the median of the sensitivity distribution. For this reason, the U.S.~Department of Defense's protocol dictates that the procedure is only appropriate for estimating the median of the distribution, and no other quantiles. This begs the question of whether the parametric model assumption can be disposed of altogether, potentially making the Langlie procedure entirely nonparametric, much like the Robbins--Monro procedure. In this paper we answer this question in the negative by proving that when the Langlie procedure is employed, the sequence of inputs converges with probability zero.
\end{abstract}


\section{Introduction}
Sensitivity testing refers to experimentally determining the susceptibility of a physical system to break, fracture or explode under various levels of stimulus. In research of energetic materials, such experiments are employed to determine the sensitivity of explosives to stimuli like friction, shock and impact \citepalias{stanag4487, stanag4488, stanag4489}. Similarly, sensitivity testing is used to estimate the probability of ignition of one-shot devices \citep{mil-std-331D}. The researcher repeatedly affects samples of the material in question with various levels of stimuli and observes which lead to explosions and which do not. Accurate sensitivity measurements are particularly important when studying materials whose hazard properties are not well documented in the literature, such as explosive remnants of war and dumped ammunition \citep{novik2022analysis, novik2024increased, christensen2024estimating, cumming2024munitions}.

Mathematically, a sensitivity test comprises $N$ binary trials, one for each stimulus, $x_1, \dots, x_N$. Given $x_n$, the binary output $y_n$ denotes success ($y_n=+1$) or failure ($y_n = -1$). Although it is more common in the sensitivity testing literature to use $y_n = 0, 1$ rather than $y_n =\pm 1$, we shall use the latter convention as it simplifies the algebra in the present paper. The pairs $(x_n, y_n)$ are thought of as realisations of random variables $(X_n, Y_n)$, for $n=1, \dots, N$. We will also write $(X, Y)$ for a generic input and outcome of a binary trial. From the observed data, we wish to estimate the underlying sensitivity distribution $F:\reals\to[0,1]$, with the defining property that $\prob(Y = +1\given X = x) = F(x)$. Also, $\prob(Y=-1\given X=x) = 1 - F(x)$. Since sensitivity testing largely is about estimating the quantiles of $F$, we introduce the notation $\xi_{100q} = F^{-1}(q)$, for $0<q<1$. In this notation, $\xi_{50}$ is the median of $F$, i.e.,~the input level yielding a $50\%$ probability of success.

A key component of sensitivity testing is the question of which experimental design to employ. That is, how to choose the inputs $X_1, \dots, X_N$. Unlike other fields like bioassay where multiple trials can be performed at once with binomial outcomes, sensitivity testing typically only allows one binary trial at a time. This motivates the employment of an adaptive design of the form
\begin{equation}\label{eq:adaptive}
    X_{n+1} = h(\left\{(X_i, Y_i)\right\}_{i=1}^{n}),
\end{equation}
for $n=1, \dots, N-1$. That is, next input depends on the data acquired up to that point. Which function $h$ should be chosen largely depends on the prior information available to the researcher and the goal of the sensitivity testing experiment. For instance, when relating the sensitivity of energetic materials to quantum chemical phenomena (either via regression or machine learning models), the median $\xi_{50}$ is of particular interest \citep{jensen2020models, michalchuk2021predicting, lansford2022building}, but in practical applications concerning safety or functionality, extreme quantiles like $\xi_{01}$ or $\xi_{99}$ are more important.

The purpose of this paper is to study one of the most popular experimental designs for sensitivity testing, namely the Langlie procedure \citep{langlie1962reliability}. This is officially recognized by the U.S.~Department of Defense for estimating the initiation probability of one-shot devices \citep{mil-std-331D}, and is thus widely used for sensitivity testing worldwide \citepalias{burke2017binary, baker2021overview}. When employing the Langlie procedure, the researcher first has to specify two parameters $a < b$, which serve as lower and upper bounds for all the trials conducted, respectively. The first trial is conducted at the input $X_1 = (a + b) / 2$, and the succeeding inputs are decided by the following inductive rule: Assume that we have conducted $n$ trials, so our observations are $(X_1, Y_1), \dots, (X_n, Y_n)$. We say there is \emph{balance at index $i$} if there are equally many successes and failures in the set $\{Y_{i+1}, \dots, Y_n\}$. That is,
\begin{equation}\label{eq:balance}
    \#\{j\in\{i+1, \dots, n\} : Y_j = -1\} = \#\{j\in\{i+1, \dots, n\} : Y_j = +1\}.
\end{equation}
Now, after $n$ trials, define $\tau_n$ to be the maximal index at which there is balance, or $0$ if there is nowhere balance. The next input $X_{n+1}$ then takes the value
\begin{equation}\label{eq:next-langlie}
    X_{n+1} = \begin{cases} (X_{\tau_n} + X_n) / 2 & \text{if}\ \tau_n > 0 \\
        (a + X_n) / 2 & \text{if}\ \tau_n = 0\ \text{and}\ Y_n = +1 \\
        (X_n + b) / 2 & \text{if}\ \tau_n = 0\ \text{and}\ Y_n = -1.
                \end{cases}
\end{equation}
In other words, if there is balance at some index $i$, then we average with the input of the maximal such index. If there is nowhere balance, then we either average with $a$ or $b$, depending on whether the last trial resulted in a success or failure.

\begin{remark}\label{rmk:sum-criterion}
    Since we decided to formulate the binary outcome $Y_n$ to be in the set $\{-1, +1\}$ rather than $\{0, 1\}$, we can reformulate criterion~\eqref{eq:balance} as $\sum_{j=i+1}^n Y_j = 0.$ Simpler still, if we define the cumulative sum process $S_n = \sum_{i=1}^n Y_i$, then criterion~\eqref{eq:balance} reads $S_i = S_n$.
\end{remark}

In the Langlie procedure, the above experimental design is combined with a parametric model, most commonly probit or logistic regression, and the sensitivity distribution is estimated using maximum likelihood. For instance, for probit regression, we would have $F(x) = \Phi(\alpha + \beta x)$, where $\alpha, \beta$ are the model parameters (to be estimated from data) and $\Phi(a) = (2\pi)^{-1/2}\int_{-\infty}^a \exp(-t^2/2)\,\diff t$ is the standard normal cumulative distribution function. This yields the log-likelihood function 
	$$\ell_n(\alpha, \beta) = \sum_{i=1}^n\left\{y_i\log\Phi(\alpha + \beta x_i) + (1 - y_i)\log[1 - \Phi(\alpha + \beta x_i)]\right\}.$$
Letting $\widehat\alpha, \widehat\beta = \argmax_{\alpha, \beta}\ell_n(\alpha, \beta)$ denote the maximum likelihood estimators, the corresponding estimator for $\xi_{50}$ is $\widehat\xi_{50} = -\widehat\alpha / \widehat\beta$. 

It has been observed by many users of the Langlie procedure that the inputs $X_n$ tend to concentrate around the median $\xi_{50}$. This is a consequence of the averaging $(X_n + X_{\tau_n}) / 2$ in the experimental design~\eqref{eq:next-langlie}. For this reason, the U.S.~Department of Defense's protocol states that the procedure should be used for estimating the median $\xi_{50}$ only, and no other quantiles of the sensitivity distribution \citep{mil-std-331D}. However, for the sole purpose of estimating $\xi_{50}$, the employment of a parametric model is unnecessarily restrictive and will lead to biased estimates if the model assumption is incorrect. Other natural alternatives to Langlie, like the Robbins--Monro procedure \citep{robbins1951stochastic}, do not suffer this restriction. Indeed, in the context of sensitivity testing, the Robbins--Monro procedure sets $X_{n+1} = X_n - a_nY_n/2$, where the non-negative sequence $(a_n)$ satisfies $\sum_{n=1}^\infty a_n^2 < \infty$ and $\sum_{n=1}^\infty a_n = \infty$. Here, the sequence $(a_n)$ is defined so as to coincide with the notation of \citet{cochran1965robbins}. \citet{robbins1951stochastic} first proved that $X_n\to\xi_{50}$ in probability, and later \citet{blum1958note} proved that $X_n\to\xi_{50}$ almost surely. That is, $\prob(X_n \to \xi_{50}\ \text{as}\ n\to\infty) = 1$. This yields a natural, model-independent estimator for $\xi_{50}$, namely $\widehat\xi_{50} = X_{N+1}$, the next hypothetical input value. Comparing the Langlie and Robbins--Monro procedures in this way begs the question of whether also the sequence $(X_n)$ converges almost surely in the Langlie procedure, so that the parametric model assumption can be disposed of altogether.

In this paper, we prove that in the Langlie procedure, $X_n$ converges with probability zero. That is, unlike the Robbins--Monro procedure, $X_{N+1}$ cannot be used as a consistent estimator for $\xi_{50}$ (or indeed any quantile $\xi_{100q})$ in the Langlie procedure.

\section{Elementary properties of the Langlie procedure}\label{sec:properties}
Before moving on the proof of non-convergence, we will briefly explore some of the elementary properties of the Langlie procedure. We begin with an exploration of the process $(X_n)$ of inputs. From \eqref{eq:next-langlie}, we see that $(X_n)$ is not a Markov chain, which makes the Langlie procedure significantly more complicated than the Robbins--Monro procedure, and indeed other popular experimental designs like the Bruceton procedure \citep{dixon1948method, christensen2024theory}. Indeed, at each index $n$, there is a positive probability that $\tau_n = 1$, meaning that the entire sequence $\{(X_1, Y_1), \dots, (X_n, Y_n)\}$ must be kept in memory in order to determine the next input $X_{n+1}$. In Figure~\ref{fig:langlieX}, we see a sample path from the process $(X_n)$, for $n=1, \dots, N$, first for $N=50$ (left) and then for $N=10^3$ (right). The indices $n$ at which $\tau_n = 0$, i.e.,~where there is nowhere balance, are marked in red. For the majority of indices, $\tau_n > 0$ and so $X_{n+1} = (X_{\tau_n} + X_n)/2$, which makes the inputs $X_n$ concentrate around the median $\xi_{50}$. However, at some of the points where $\tau_n = 0$, $X_{n+1}$ is obtained by averaging $X_n$ with $a$ or $b$, and so there is a significant jump in the process, before it quickly concentrates around $\xi_{50}$ again. There are also other indices where $\tau_n > 0$ but where there is also a minor jump in the value of $X_{n+1}$.

\begin{figure}
    \begin{center}
        \includegraphics[scale=.65]{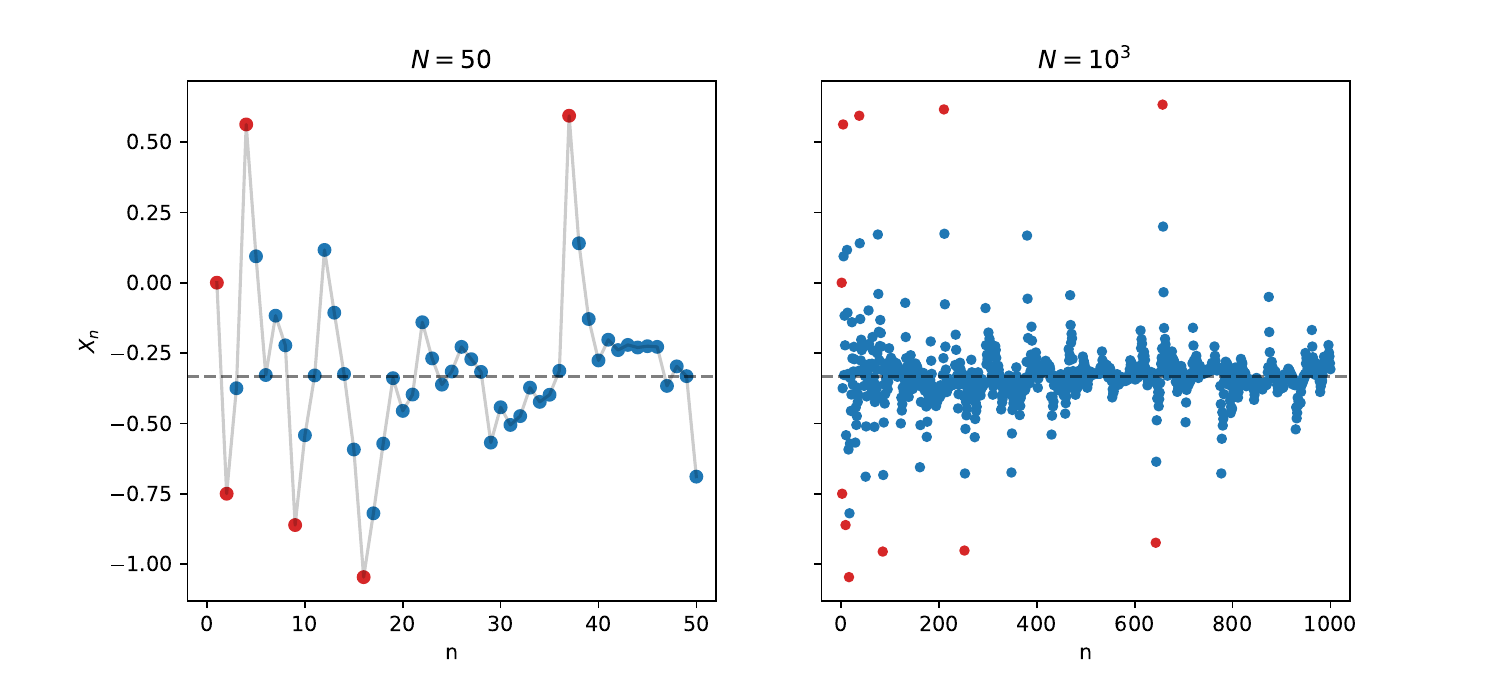}
    \end{center}
    \caption{The process $(X_n)$, for $n=1, \dots, N$, where $N=50$ (left) and $N=10^3$ (right). The process was generated using $(a, b) = (-1.5, 1.5)$, and the true sensitivity distribution was $F(x) = \Phi(\alpha + \beta x)$, with $(\alpha, \beta) = (3.333, 9.999)$. The true median $\xi_{50} = -\alpha/\beta = -0.333$ is shown by the dashed line, and the points $(n, X_n)$ at which $\tau_n = 0$ are colored in red. \label{fig:langlieX}}
\end{figure}

Further insight into the mechanisms of the Langlie procedure are obtained by studying the process $(Y_n)$. More precisely, we will focus on the cumulative sums $S_n = \sum_{i=1}^n Y_i$, defined on the integers. This is a non-symmetric random walk with dependent increments, which makes it more complicated to study than a classical random walk in which the increments are assumed to be independent. In Figure~\ref{fig:langlieY}, we see the sample paths for $(S_n)$, for $n=1, \dots, N$, for $N=50$ and $N=10^3$, corresponding to those in Figure~\ref{fig:langlieX}. A crucial observation can be made from this figure, namely that the indices $n$ for which $\tau_n = 0$ are precisely those where the process $(S_n)$ visits a new value not previously attained. We prove this rigorously.

\begin{figure}
    \begin{center}
        \includegraphics[scale=.65]{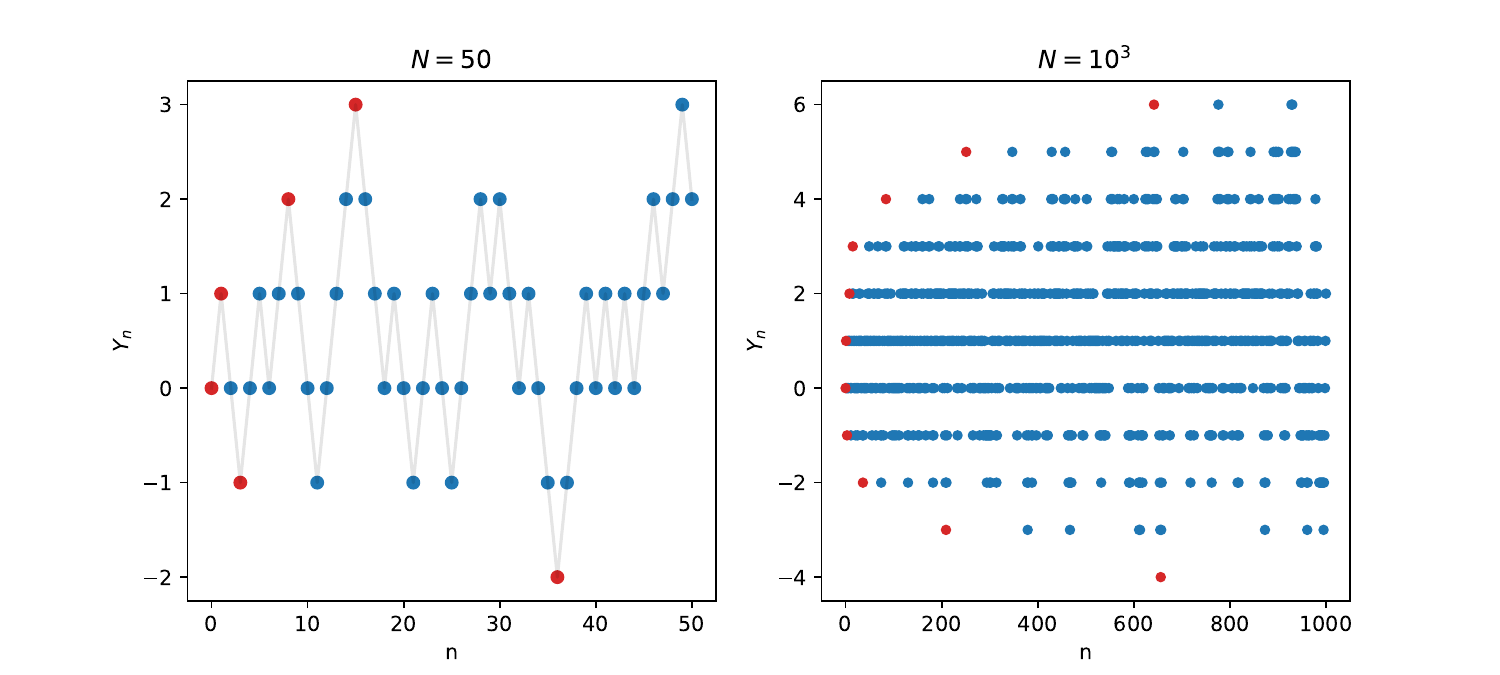}
    \end{center}
    \caption{The process $(S_n)$, for $n=1, \dots, N$, where $N=50$ (left) and $N=10^3$ (right), corresponding to the $(X_n)$ paths shown in Figure~\ref{fig:langlieX}. The points $(n, S_n)$ at which $\tau_n = 0$ are colored in red. \label{fig:langlieY}}
\end{figure}

\begin{proposition}\label{prop:bounded-Y}
    Let $(X_n, Y_n)_{n\geq 1}$ follow the Langlie procedure with parameters $a < b$, and let $S_n = \sum_{i=1}^n Y_i$. Then for all $n\geq 2$, $\tau_n = 0$ if and only if $S_n \notin \{S_1, \dots, S_{n-1}\}$.
\end{proposition}

\begin{proof}
    If $S_n \in \{S_1, \dots, S_{n-1}\}$, then $S_n = S_i $ for some $i\in\{1, \dots, n-1\}$, and so $\tau_n = i$ by Remark~\ref{rmk:sum-criterion}. As all implications also hold in reverse, the converse is also true.
\end{proof}

As will be made clear in the next section, studying the prevalence of indices $n$ for which $\tau_n = 0$ is the key for establishing non-convergence of the process $(X_n)$.

\section{Proof of non-convergence}\label{sec:proof}
In this section, we prove that for all $q\in(0, 1)$ such that $a < F^{-1}(q) < b$, $\prob(X_n\to\xi_{100q}\ \text{as}\ n\to\infty) = 0$, thus demonstrating that the Langlie procedure may not be turned into a nonparametric estimation procedure like Robbins--Monro. The key to the proof will be to compare the process $(S_n)$ with a suitable random walk with independent increments. We will compare these processes using the stochastic order, defined as follows.

\begin{definition}
    Let $P$ and $Q$ be two real-valued random variables. We say that $P$ is \emph{smaller than $Q$ in the usual stochastic order}, denoted by $P\leqst Q$, if $\prob(P \leq x) \geq \prob(Q \leq x)$ for all $x\in\reals$. We say that $P$ and $Q$ are \emph{equal in law}, denoted by $P =_{\mathrm{st}} Q$, if $P\leqst Q$ and $Q\leqst P$. That is, $P=_{\mathrm{st}}Q$ if and only if $\prob(P \leq x) = \prob(Q\leq x)$ for all $x\in\reals$.
\end{definition}

Using the usual stochastic order, the next lemma allows us to compare the sums $S_n$ associated to the Langlie procedure with a simpler random walk with independent increments.

\begin{lemma}\label{lemma:stochastic-order}
    Let $Y_1, Y_2, \dots$ be (dependent) random variables taking values in $\{-1, +1\}$, with $\prob(Y_1 = +1) = f_1$ and $\prob(Y_{n+1} = +1\given Y_1, \dots, Y_n) = f_{n+1}(Y_1, \dots, Y_n)$ for $n\geq 1$. Here, $0 < f_1 < 1$ is a constant and $f_{n+1}$ are measurable functions for all $n$. Also, write $S_n = \sum_{i=1}^n Y_i$ and let $A_n = |S_n|$. Let $$\mathcal{S} = \{f_1\}\cup\bigcup_{n=1}^\infty\bigcup_{y_1, \dots, y_n=\pm 1}\{f_{n+1}(y_1, \dots, y_n)\},$$
    and write $(1 - \mathcal{S})$ for the set $\{1 - s : s \in \mathcal{S}\}$. Assume that the set $\mathcal{S}\cup(1 - \mathcal{S})$ is uniformly bounded below by a constant $p$, with $0 < p < 1/2$. Next, let $Z_1, Z_2, \dots$ be independent random variables taking values in $\{-1, +1\}$, with $\prob(Z_n = +1) = p$ for $n \geq 1$. Define the variables $B_1, B_2, \dots$ by $B_1 = 1$ and $B_{n+1} = |B_n + Z_{n+1}|$ for $n\geq 1$. That is, $(B_n)$ is a non-negative random walk with independent increments $Z_n$ and a strong reflective barrier at the origin. Finally, let $P_n = \max\{A_1, \dots, A_n\}$ and $Q_n = \max\{B_1, \dots, B_n\}$. Then $Q_n \leq_{\mathrm{st}} P_n$ for all $n\geq 1$.
\end{lemma}

\begin{proof}
    Fix an index $n$, and let $U_1, \dots, U_n \sim\mathrm{Uniform}[0, 1]$ be independent uniform random variables. Now let $\tilde{Y}_1 = 2\one\{U_1\leq f_1\} - 1$ and inductively,
        $$\tilde{Y}_{n+1} = 2\one\{U_{n+1} \leq f_{n+1}(\tilde{Y}_1, \dots, \tilde{Y}_n)\} - 1$$
    for $n\geq 1$. Next, let $\tilde{S}_n = \sum_{i=1}^n \tilde{Y}_i$ and $\tilde{A}_n = |\tilde{S}_n|$. Then $\tilde{P}_n = \max\{\tilde{A}_1, \dots, \tilde{A}_n\}$ is a function of $U_1, \dots, U_n$ satisfying $\tilde{P}_n =_{\mathrm{st}} P_n$. Similarly, let $\tilde{Z}_n = 2\one\{U_n \leq p\} - 1$ for all $n$, and define the random variables $\tilde{B}_1, \tilde{B}_2, \dots$ by $\tilde{B}_1 = 1$ and $\tilde{B}_{n+1} = |\tilde{B}_n + \tilde{Z}_{n+1}|$ for $n\geq 1$. Then $\tilde{Q}_n = \max\{\tilde{B}_1, \dots, \tilde{B}_n\}$ is a function of $U_1, \dots, U_n$ satisfying $\tilde{Q}_n =_{\mathrm{st}} Q_n$. Crucially, $\tilde{P}_n$ and $\tilde{Q}_n$ are defined on the same probability space. Now, since $p$ is a lower bound for the set $\mathcal{S}\cup(1 - \mathcal{S})$, we have that $\prob(\tilde{B}_n \leq \tilde{A}_n) = 1$, and consequently, $\prob(\tilde{Q}_n \leq \tilde{P}_n) = 1$. Now apply Theorem 1.A.1 by \citet{shaked2007stochastic}.
\end{proof}

Next, we need a lemma which links the usual stochastic order with the property of boundedness.

\begin{lemma}\label{lemma:technical}
    Let $(A_n)$ and $(B_n)$ be two discrete stochastic processes taking values in $\{0, 1, \dots\}$, not necessarily with independent increments, and let $P_n = \max\{A_1, \dots, A_n\}$ and $Q_n = \max\{B_1, \dots, B_n\}$ for $n\geq 1$. If $Q_n \leqst P_n$ for all $n$, then
        $$\prob((A_n)\ \mathrm{is\ bounded}) \leq \prob((B_n)\ \mathrm{is\ bounded}).$$
\end{lemma}

\begin{proof}
    Define the events $V_m = \bigcap_{n\geq 1}\{P_n \leq m\}$. Then 
        $$V_m = \{P_n \leq m\ \text{for}\ \text{all}\ n\geq 1\} \subseteq \{P_n \leq m+1\ \text{for}\ \text{all}\ n\geq 1\} = V_{m+1},$$
        so $V_1 \subseteq V_2 \subseteq \dots$ is an ascending sequence of events. Also, for a fixed number $m$, define $W_n^{(m)} = \{P_n \leq m\}^C = \{P_n > m\}$, and note that
        $$W_n^{(m)} = \{A_j > m\ \text{for}\ \text{some}\ 1 \leq j\leq n\} \subseteq \{A_j > m\ \text{for}\ \text{some}\ 1\leq j \leq n+1\} = W_{n+1}^{(m)},$$
    so $W_1^{(m)}\subseteq W_2^{(m)}\subseteq \dots$ is ascending (in $n$). Hence, by continuity of measure and de~Morgan's law,
    \begin{multline*}
        \prob\left((A_n)\ \text{is\ bounded}\right) = \prob\left(\bigcup_{m\geq 1}V_m\right) = \lim_{m\to\infty}\prob(V_m) = 1 - \lim_{m\to\infty}\prob\left(\bigcup_{n\geq 1} W_n^{(m)}\right) \\
        = \lim_{m\to\infty}\lim_{n\to\infty}\prob(P_n\leq m) \leq\lim_{m\to\infty}\lim_{n\to\infty}\prob(Q_n\leq m).
    \end{multline*}
    Now reverse every step in the above derivation to obtain the desired inequality.
\end{proof}

Finally, before we can move on to the main result, we need to establish a simple fact about reflected random walks.

\begin{lemma}\label{lemma:random-walk}
    Let $p\in (0, 1)$ with $p<1/2$, and let $(B_n)$ be a reflected random walk as described in Lemma~\ref{lemma:stochastic-order} with $p$ as its parameter. Let $m > 0$ be any positive integer. Then 
        $$\prob(B_n = m\ \mathrm{infinitely\ often}) = \prob\left(\bigcap_{k=1}^\infty\bigcup_{n=k}^\infty\{B_n = m\}\right) = 1.$$
    In particular, $\prob((B_n)\mathrm{\ is\ bounded}) = 0$.
\end{lemma}

\begin{proof}
    As $p < 1/2$, this is immediate from the fact that $B_n$ is an irreducible and positive recurrent Markov chain whose stationary distribution has the non-negative integers as its support.
\end{proof}

We are now ready to prove the main result.

\begin{theorem}\label{thm:master}
    Let $(X_n, Y_n)_{n\geq 1}$ follow the Langlie procedure for some underlying distribution function $F$ and initial parameters $a < b$. For any $q\in (0, 1)$ with $a < \xi_{100q} = F^{-1}(q) < b$, we have that $\prob(X_n\to\xi_{100q}\mathrm{\ as\ }n\to\infty) = 0$.
\end{theorem}

\begin{proof}
    Fix $q$ and $\xi = F^{-1}(q)$. The first step of the proof is to establish that
        \begin{align}
        \{X_n\to\xi\mathrm{\ as\ }n\to\infty\} & \subseteq \{\tau_n = 0\mathrm{\ finitely\ many\ times}\}\notag \\
        & = \{\tau_n = 0\mathrm{\ infinitely\ often}\}^C\label{eq:claim}.
        \end{align}
    To prove this, let $(x_n, y_n)_{n\geq 1}$ be a realisation of the Langlie procedure and suppose that $x_n\to\xi$ as $n\to\infty$. Then in particular, for $\epsilon < \min\{(\xi - a)/2, (b - \xi)/2\}$, we know that $|x_n - \xi| < \epsilon$ for sufficiently large $n$. But for any such large $n$ with $\tau_n = 0$, we see that if $y_n = -1$, then
        $$|x_{n+1} - \xi| = |(x_n + b)/2 - \xi| \geq \frac12\left||b - \xi| - |x_n - \xi|\right| = \frac12|b - \xi| - \frac12|x_n - \xi| > \frac14|b - \xi| > \epsilon,$$
    violating convergence. Similarly, $|x_{n+1} - \xi| > \epsilon$ if $\tau_n = 0$ and $y_n = +1$. This means that $\tau_n = 0$ only finitely many times, proving \eqref{eq:claim}.

    As before, let $S_n = \sum_{i=1}^n Y_i$ and $A_n = |S_n|$. Then $(S_n)$ is bounded if and only if $(A_n)$ is bounded. Let $0 < p < \min\{F(a), 1 - F(a), F(b), 1 - F(b)\} < 1/2$ and let $(B_n)$ be a reflected random walk as described in Lemma~\ref{lemma:stochastic-order} with $p$ as its parameter. By Lemma~\ref{lemma:random-walk}, $\prob((B_n)\ \mathrm{is}\ \mathrm{bounded}) = 0$. Further, by Lemma~\ref{lemma:stochastic-order} and Lemma~\ref{lemma:technical}, $\prob((A_n)\ \mathrm{is}\ \mathrm{bounded}) \leq \prob((B_n)\ \mathrm{is}\ \mathrm{bounded})$. Putting everything together, we get
    \begin{align*}
        \prob(X_n\to\xi\ \mathrm{as}\ n\to\infty) & \leq \prob(\tau_n = 0\ \mathrm{finitely\ many\ times}) && \mathrm{by}\ \eqref{eq:claim} \\
        & = \prob((S_n)\ \mathrm{is\ bounded}) && \mathrm{by}\ \mathrm{Proposition}~\eqref{prop:bounded-Y} \\
        & = \prob((A_n)\ \mathrm{is\ bounded}) && \\
        & \leq \prob((B_n)\ \mathrm{is\ bounded}) && \mathrm{by}\ \mathrm{Lemmas}~\ref{lemma:stochastic-order}\ \mathrm{and}~\ref{lemma:technical} \\
        & = 0 && \mathrm{by}\ \mathrm{Lemma}~\ref{lemma:random-walk},
    \end{align*}
    as required.
\end{proof}

For the purpose of sensitivity testing, Theorem~\ref{thm:master} tells us that $X_{N+1}$ is not a consistent estimator for $\xi_{50}$, given data $\{(X_n, Y_n)\}_{n=1}^N$. From the proof of the theorem, we see that there is no theoretical guarantee that $X_{n+1}$ is close to $\xi_{50}$, even for arbitrarily large values of $n$. Indeed, the probability that $\tau_n = 0$ is always non-negligible. Thus, the parametric modeling must be viewed as an intrinsic component of the Langlie procedure. Hence, unlike model-independent alternatives like the Robbins--Monro procedure, the Langlie procedure should only be applied when user is confident that the parametric model is correctly specified.

\bibliographystyle{plainnat}
\bibliography{bibliography}

\begin{thebibliography}{20}
\providecommand{\natexlab}[1]{#1}
\providecommand{\url}[1]{\texttt{#1}}
\expandafter\ifx\csname urlstyle\endcsname\relax
  \providecommand{\doi}[1]{doi: #1}\else
  \providecommand{\doi}{doi: \begingroup \urlstyle{rm}\Url}\fi

\bibitem[Baker(2021)]{baker2021overview}
E.~L. Baker.
\newblock An overview of single event testing - methods and analysis.
\newblock Technical Report O-216, Munitions Safety Information Analysis Center (MSIAC), 2021.

\bibitem[Blum(1958)]{blum1958note}
J.~R. Blum.
\newblock A note on stochastic approximation.
\newblock \emph{Proceedings of the American Mathematical Society}, 9:\penalty0 404--407, 1958.

\bibitem[Burke and Truett(2017 ({R}evised 2018))]{burke2017binary}
S.~Burke and L.~Truett.
\newblock Binary response and single stress factor test methods.
\newblock Technical Report 08-2017, Air Force Institute of Technology: STAT Center of Excellence, 2017 ({R}evised 2018).

\bibitem[Christensen et~al.(2024{\natexlab{a}})Christensen, Novik, and Unneberg]{christensen2024estimating}
D.~Christensen, G.~P. Novik, and E.~Unneberg.
\newblock Estimating sensitivity with the {B}ruceton method: {S}etting the record straight.
\newblock \emph{Propellants, Explosives, Pyrotechnics}, In press, 2024{\natexlab{a}}.
\newblock \doi{10.1002/prep.202400022}.

\bibitem[Christensen et~al.(2024{\natexlab{b}})Christensen, Stoltenberg, and Hjort]{christensen2024theory}
D.~Christensen, E.~A. Stoltenberg, and N.~L. Hjort.
\newblock Theory for adaptive designs in regression.
\newblock \emph{Submitted}, 2024{\natexlab{b}}.

\bibitem[Cochran and Davis(1965)]{cochran1965robbins}
W.~G. Cochran and M.~Davis.
\newblock The {R}obbins--{M}onro method for estimating the median lethal dose.
\newblock \emph{Journal of the Royal Statistical Society. Series B (Methodological)}, 27:\penalty0 28--44, 1965.
\newblock \doi{10.1111/j.2517-6161.1965.tb00583.x}.

\bibitem[Cumming(2024)]{cumming2024munitions}
A.~Cumming.
\newblock Munitions underwater - a problem for today.
\newblock \emph{Propellants, Explosives, Pyrotechnics}, 49, 2024.
\newblock \doi{10.1002/prep.202400052}.

\bibitem[Dixon and Mood(1948)]{dixon1948method}
W.~J. Dixon and A.~M. Mood.
\newblock A method for obtaining and analyzing sensitivity data.
\newblock \emph{Journal of the American Statistical Association}, 43:\penalty0 109--126, 1948.
\newblock \doi{10.1080/01621459.1948.10483254}.

\bibitem[Jensen et~al.(2020)Jensen, Moxnes, Unneberg, and Christensen]{jensen2020models}
T.~L. Jensen, J.~F. Moxnes, E.~Unneberg, and D.~Christensen.
\newblock Models for predicting impact sensitivity of energetic materials based on the trigger linkage hypothesis and {A}rrhenius kinetics.
\newblock \emph{Journal of Molecular Modeling}, 26, 2020.
\newblock \doi{10.1007/s00894-019-4269-z}.

\bibitem[Langlie(1962)]{langlie1962reliability}
H.~J. Langlie.
\newblock A reliability test for ``one-shot'' items.
\newblock Technical Report U-1792, Aeronutronic Division, Ford Motor Company, Newport Beach, CA, 1962.

\bibitem[Lansford et~al.(2022)Lansford, Barnes, Rice, and Jensen]{lansford2022building}
J.~L. Lansford, B.~C. Barnes, B.~M. Rice, and K.~F. Jensen.
\newblock Building chemical property models for energetic materials from small datasets using a transfer learning approach.
\newblock \emph{Journal of Chemical Information and Modeling}, 62:\penalty0 5397--5410, 2022.
\newblock \doi{10.1021/acs.jcim.2c00841}.

\bibitem[Michalchuk et~al.(2021)Michalchuk, Hemingway, and Morisson]{michalchuk2021predicting}
A.~A.~L. Michalchuk, J.~Hemingway, and C.~A. Morisson.
\newblock Predicting the impact sensitivities of energetic materials through zone-center phonon up-pumping.
\newblock \emph{The Journal of Chemical Physics}, 154, 2021.
\newblock \doi{10.1063/5.0036927}.

\bibitem[{North Atlantic Treaty Organization (NATO)}(1999)]{stanag4489}
{North Atlantic Treaty Organization (NATO)}.
\newblock {STANAG} 4489 - {E}xplosives, impact sensitivity tests, 1999.

\bibitem[{North Atlantic Treaty Organization (NATO)}(2009{\natexlab{a}})]{stanag4487}
{North Atlantic Treaty Organization (NATO)}.
\newblock {STANAG} 4487 - {E}xplosives, friction sensitivity tests, 2009{\natexlab{a}}.

\bibitem[{North Atlantic Treaty Organization (NATO)}(2009{\natexlab{b}})]{stanag4488}
{North Atlantic Treaty Organization (NATO)}.
\newblock {STANAG} 4488 - {E}xplosives, shock sensitivity tests, 2009{\natexlab{b}}.

\bibitem[Novik(2022)]{novik2022analysis}
G.~P. Novik.
\newblock Analysis of samples of high explosives extracted from explosive remnants of war.
\newblock \emph{Science of the Total Environment}, 842, 2022.
\newblock \doi{10.1016/j.scitotenv.2022.156864}.

\bibitem[Novik and Christensen(2024)]{novik2024increased}
G.~P. Novik and D.~Christensen.
\newblock Increased impact sensitivity in ageing high explosives; analysis of {A}matol extracted from explosive remnants of war.
\newblock \emph{Royal Society Open Science}, 11, 2024.
\newblock \doi{10.1098/rsos.231344}.

\bibitem[Robbins and Monro(1951)]{robbins1951stochastic}
H.~Robbins and S.~Monro.
\newblock A stochastic approximation method.
\newblock \emph{The Annals of Mathematical Statistics}, 22:\penalty0 400--407, 1951.
\newblock \doi{10.1214/aoms/1177729586}.

\bibitem[Shaked and Shanthikumar(2007)]{shaked2007stochastic}
M.~Shaked and J.~G. Shanthikumar.
\newblock \emph{Stochastic Orders}.
\newblock Springer, 2007.

\bibitem[{U.S.~Department of Defense}(2017)]{mil-std-331D}
{U.S.~Department of Defense}.
\newblock {MIL}-{STD}-331{D} ({A}ppendix {G}) - {S}tatistical methods to determine the initiation probability of one-shot devices, 2017.

\end{thebibliography}

\end{document}